\documentclass[12pt]{article}
\usepackage{amsmath,amsfonts,amssymb}
\usepackage{array}
\usepackage{amsthm}
\usepackage{graphicx, epstopdf}
\usepackage{lscape} 

\newcommand{\la}{\ensuremath{\langle}}
\newcommand{\ra}{\ensuremath{\rangle}}

\newtheorem{thm}{Theorem}[section]
\newtheorem{lemma}[thm]{Lemma}
\newtheorem{prop}[thm]{Proposition}

\newtheorem{cor}[thm]{Corollary}

\newtheorem{prog}[thm]{Program}

\usepackage{graphicx}
\newcommand{\coverpage}[3]{\thispagestyle{empty}
    \addtocounter{page}{-1}
\null\vspace*{-1cm} \hfill\includegraphics[scale=0.2]{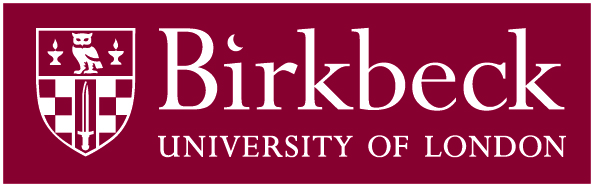} \vskip 2in
\begin{center} \begin{minipage}{0.7\textwidth}\begin{center}\Huge\bf{#1}\end{center} \end{minipage}\end{center}  \vfill
\begin{center} {\large By}\bigskip\\ {\large #2}\\ \end{center} \vfill 
 \framebox{\begin{minipage}{\textwidth}
Birkbeck Pure Mathematics Preprint Series\hfill
Preprint Number #3 \\ \\
\null\hfill www.bbk.ac.uk/ems/research/pure/preprints 
\end{minipage}}
\newpage}

\begin{document}

\coverpage{Locally maximal product-free sets of size 3}{Chimere S. Anabanti and Sarah B. Hart}{10}

\title{Locally maximal product-free sets of size 3}
\author{Chimere S. Anabanti\thanks{The first author is supported by a Birkbeck PhD Scholarship}\\ c.anabanti@mail.bbk.ac.uk \and Sarah B. Hart\\ s.hart@bbk.ac.uk}  

\date{April 2015}

\maketitle

\begin{abstract}
\noindent Let $G$ be a group, and $S$ a non-empty subset of $G$. Then $S$ is \emph{product-free} if $ab\notin S$ for all $a, b \in S$. We say $S$ is \emph{locally maximal product-free} if $S$ is product-free and not properly contained in any other product-free set. A natural question is what is the smallest possible size of a locally maximal product-free set in $G$. The groups containing locally maximal product-free sets of sizes $1$ and $2$ were classified in \cite{gh}. In this paper, we prove a conjecture of Giudici and Hart in \cite{gh} by showing that if $S$ is a locally maximal product-free set of size $3$ in a group $G$, then $|G| \leq 24$. This shows that the list of known locally maximal product-free sets given in \cite{gh} is complete.
\end{abstract}

\section{Introduction}

\noindent Let $G$ be a group, and $S$ a non-empty subset of $G$. Then $S$ is
\emph{product-free} if $ab\notin S$ for all $a, b \in S$. For example,
if $H$ is a subgroup of $G$ then $Hg$ is a product-free set for any
$g\notin H$. Traditionally these sets have been studied in abelian groups, and have therefore been called sum-free sets. Since we are working with arbitrary groups it makes more sense to say `product-free' in this context. We say $S$ is \emph{locally maximal product-free} if $S$ is
product-free and not properly contained in any other product-free set.
We use the term {\em locally maximal} rather than maximal because 
the majority of the literature in this area uses {\em maximal} to mean maximal by cardinality (for example \cite{streetwhite,SW1974A}).\\

There are some obvious questions from the definition: given a group $G$, what is the maximum cardinality of a product-free set in $G$, and what are the maximal (by cardinality) product-free sets? How many product-free sets are there in $G$? Given that each product-free set is contained in a locally maximal product-free set, what are the locally maximal product-free sets? What are the possible sizes of locally maximal product-free sets? The question of maximal (by cardinality) product-free sets has been fully solved for abelian groups by Green and Rusza \cite{greenruzsa}. For the nonabelian case Kedlaya \cite{kedlaya97} showed that there exists a constant $c$ such that the largest product-free set in a group of order $n$ has size at least $cn^{11/14}$. Gowers \cite{gowers} proved that if the smallest nontrivial representation of $G$ is of dimension $k$ then the largest product-free set in $G$ has size at most $k^{-1/3}n$ (Theorem 3.3 and commentary at the start of Section 5). Much less is known about the minimum sizes of locally maximal product-free sets. This question was first asked in \cite{babaisos} where the authors ask what is the minimum size of a locally maximal product-free set in a group of order $n$? A good bound is still not known. Small locally-maximal product-free sets when $G$ is an elementary abelian 2-group are of interest in finite geometry, because they correspond to complete caps in PG($n-1,2$). In \cite{gh}, the groups containing locally maximal product-free sets of sizes $1$ and $2$ were classified. Some general results were also obtained. Furthermore, there was a classification (Theorem $5.6$) of groups containing locally maximal product-free sets $S$ of size $3$ for which not every subset of size $2$ in $S$ generates $\la S \ra$. Each of these groups has order at most $24$. Conjecture $5.7$ of \cite{gh} was that if $G$ is a group of order greater than $24$, then $G$ does not contain a locally maximal product-free set of size $3$. Table $5$ listed all the locally maximal product-free sets in groups of orders up to $24$. So the conjecture asserts that this list is the complete list of all such sets. We have reproduced Table $5$ as Table \ref{tab1} in this paper because we need to use it in some of the arguments here. The main result of this paper is the following and its immediate corollary.

\begin{thm}
\label{msf3}
Suppose $S$ is a locally maximal product-free set of size 3 in a group $G$, such that every two element subset of $S$ generates $\langle S\rangle$. Then $|G| \leq 24$.
\end{thm} 

\begin{cor}
If a group $G$ contains a locally maximal product-free set $S$ of size 3, then $|G| \leq 24$ and the only possibilities for $G$ and $S$ are listed in Table~\ref{tab1}.
\end{cor}

\begin{proof}
If not every two-element subset of $S$ generates $\langle S\rangle$, then by Theorem $5.6$ of \cite{gh}, $|G| \leq 24$. We may therefore assume that every two-element subset of $S$ generates $\langle S\rangle$. Then $|G| \leq 24$ by Theorem \ref{msf3}. Now Table \ref{tab1} is just Table $5$ of \cite{gh}; it is a list of all locally maximal product-free sets of size $3$ occurring in groups of order up to $24$ (in fact, up to $37$ in the original paper). Since we have shown that all locally maximal product-free sets of size 3 occur in groups of order up to $24$, this table now constitutes a complete list of possibilities. 
\end{proof}

\noindent We finish this section by establishing the notation to be used in the rest of the paper, and giving some basic results from \cite{gh}. For subsets $A, B$ of a group $G$, we use the standard notation $AB$ for the product of $A$ and $B$. That is,
$$AB = \{ab : a \in A, b \in B\}.$$
%
\noindent By definition, a nonempty set $S \subseteq G$ is product-free if and
only if $S \cap SS = \varnothing$. In order to investigate locally maximal
product-free sets, we introduce some further notations. For a set $S \subseteq G$, we define the following sets:
\vspace*{-1mm}
\begin{eqnarray*}
S^2 &=& \{a^2: a \in S\};\\
S^{-1} &=& \{a^{-1}: a \in S\};\\
\sqrt S &=& \{x \in G: x^2 \in S\};\\
T(S) &=& S \cup SS \cup SS^{-1} \cup S^{-1}S;\\
\hat S &=& \{s \in S : \sqrt{\{s\}}\not\subset \la S
\ra\}.
\end{eqnarray*}

\noindent For a singleton set $\{a\}$, we usually write $\sqrt a$ instead of $\sqrt{\{a\}}$.\\

\noindent For a positive integer $n$, we will denote by $\mathrm{Alt}(n)$ the alternating group of degree $n$, by $C_n$ the cyclic group of order $n$, by $D_{2n}$ the dihedral group of order $2n$, and by $Q_{4n}$ the dicyclic group of order $4n$ given by $Q_{4n}:= \langle x,y: x^{2n} = 1, x^n = y^2, yx = x^{-1}y\rangle$.\\

We finish this section with a few results from \cite{gh}.

\begin{lemma}\label{3.1}
\cite[Lemma 3.1]{gh} Suppose $S$ is a product-free set in the group $G$. Then $S$ is locally maximal product-free if and only if $G = T(S) \cup \sqrt S$.
\end{lemma}

The next result lists, in order, Proposition 3.2, Theorem 3.4, Propositions 3.6, 3.7, 3.8 and Corollary 3.10 of \cite{gh}.
\begin{thm}\label{key} Let $S$ be a locally maximal product-free set in a group $G$. Then 
\begin{enumerate}
\item[(i)] $\la S \ra$ is a normal
subgroup of $G$ and $G/\la S \ra$ is either trivial or
an elementary abelian 2-group;

\item[(ii)] $|G| \leq 2|T(S)|\cdot|\la S \ra|$;
\item[(iii)] if $\la S \ra$ is not an elementary abelian 2-group and $|\hat S| =
1$, then $|G| = 2|\la S \ra|$;
\item[(iv)] every
element $s$ of $\hat S$ has even order, and all odd powers of
$s$ lie in $S$;
\item[(v)] if there exists $s \in S$ and integers $m_1, \ldots, m_t$ such
that $\hat S = \{s, s^{m_1}, \ldots, s^{m_t}\},$ then $|G|$
divides $4|\la S\ra|$;
\item[(vi)] if $S\cap S^{-1} = \varnothing$, then $|G| \leq 4|S|^2+1$.
\end{enumerate}
\end{thm}

We require one final result.
\begin{thm}\cite[Theorem 5.1]{gh} \label{5.1}
Up to isomorphism, the only instances of
locally  maximal product-free sets $S$ of size 3 of a group $G$ where
$|G| \leq 37$ are given in Table \ref{tab1}.
\end{thm}

\section{Proof of Theorem \ref{msf3}}

\begin{prop}\label{pro1}
Suppose $S$ is locally maximal product-free of size 3 in $G$. If $\langle S\rangle$ is cyclic, then $|G| \leq 24$. 
\label{cyclic}
\end{prop}

\begin{proof} Write $S = \{a,b,c\}$. 
First note that since $\langle S\rangle$ is abelian, $SS^{-1} = S^{-1}S$; moreover $aa^{-1} = bb^{-1} = cc^{-1} = 1$; so $|SS^{-1}| \leq 7$. Also $SS \subseteq \{a^2, b^2, c^2, ab, ac, bc\}$. Thus $$|T(S)| = |S \cup SS \cup SS^{-1}| \leq 3 + 6 + 7 = 16.$$ By Lemma \ref{3.1}, $G = T(S) \cup \sqrt S$; so $\langle S \rangle = T(S) \cup (\langle S\rangle \cap \sqrt S)$. Elements of cyclic groups have at most two square roots. Therefore $|\langle S\rangle | \leq 16 + 6 = 22$. By Table \ref{tab1}, $\langle S\rangle$ must now be one of $C_6$, $C_8$, $C_9$, $C_{10}$, $C_{11}$, $C_{12}$, $C_{13}$ or $C_{15}$. Theorem \ref{key}(iv) tells us that every element $s$ of $\hat S$ has even order and all odd powers of $s$ lie in $S$. This means that for $C_{9}$, $C_{11}$, $C_{13}$ or $C_{15}$, we have $\hat S = \varnothing$ and so $G = \langle S \rangle$. In particular, $|G|\leq 24$.\\

 It remains to consider $C_6$, $C_8$, $C_{10}$ and $C_{12}$. For $C_6 = \langle g: g^6 = 1\rangle$, the unique locally maximal product-free set of size $3$ is $S = \{g, g^3, g^5\}$. Now if $g$ or $g^5$ is contained in $\hat S$, then $\hat S$ consists of powers of a single element; so by Theorem~\ref{key}(v), $|G|$ divides $24$. If neither $g$ nor $g^5$ is in $\hat S$, then $|\hat S| \leq 1$, and so by Theorem~\ref{key}(iii) therefore, $|G|$ divides $12$. In $C_8$ there is a unique (up to group automorphisms) locally maximal product-free set of size $3$, and it is $\{g, g^{-1}, g^4\}$, where $g$ is any element of order $8$. If $\hat S$ contains $g$ or $g^{-1}$, then $S$ contains all odd powers of that element by Theorem~\ref{key}(iv), and hence $S$ contains $\{g, g^3, g^5, g^7\}$, a contradiction. Therefore $|\hat S| \leq 1$ and so $|G|$ divides $16$. Next, we consider $\langle S \rangle = C_{10}$. Recall that elements of $\hat S$ must have even order. If $\hat S$ contains any element of order 10, then $S$ contains all five odd powers of this element, which is impossible by Theorem~\ref{key}(iv). This leaves only the involution of $C_{10}$ as a possible element of $\hat S$. Hence again $|\hat S|\leq 1$ and $|G|$ divides $20$. Finally we look at $C_{12}$. If $\hat S$ contains any element of order $12$, then $|S| \geq 6$, a contradiction. If $\hat S$ contains an element $x$ of order 6 then $S$ contains all three of its odd powers, so $S = \{x, x^3, x^5\}$. But then $\langle S \rangle \cong C_6$, contradicting the assumption that $\langle S \rangle = C_{12}$. Therefore, $\hat S$ can only contain elements of order $2$ or $4$. Up to group automorphism, we see from Table \ref{tab1} that every locally maximal product-free set $S$ of size $3$ in $C_{12}$ with $\langle S \rangle = C_{12}$ is one of $\{g,g^6,g^{10}\}$ or $\{g,g^3,g^8\}$ for some generator $g$ of $C_{12}$. Each of these sets contains exactly one element of order $2$ or $4$. Therefore in every case, $|\hat S| \leq 1$ and so $|G|$ divides $24$. This completes the proof.
\end{proof}

Note that the bound on $|G|$ in Proposition \ref{pro1} is attainable. For example in $Q_{24}$ there is a locally maximal product-free set $S$ of size $3$, with $\langle S \rangle \cong C_{12}$.

\begin{prop}\label{1inv}
Suppose $S$ is locally maximal product-free of size $3$ in $G$ such that every $2$-element subset of $S$ generates $\langle S \rangle$. Then either $|G| \leq 24$ or $S$ contains exactly one involution.
\end{prop}

\begin{proof}
First suppose $S$ contains no involutions. If $S \cap S^{-1} = \varnothing$, then Theorem \ref{key}(vi) tells us that $G$ has order at most 37, and then by Theorem \ref{5.1}, $(G,S)$ is one of the possibilities listed in Table \ref{tab1}. In particular $|G| \leq 24$. If $S \cap S^{-1} \neq \varnothing$, then $S = \{a, a^{-1}, b\}$ for some $a, b$. But then $\langle S \rangle = \langle a, a^{-1}\rangle = \langle a \rangle$, so $\langle S \rangle$ is cyclic. Now by Proposition \ref{cyclic} we get $|G| \leq 24$. Next, suppose that $S$ contains at least two involutions, $a$ and $b$, with the third element being $c$. Then, since every 2-element subset of $S$ generates $\langle S \rangle$, we have that $H = \langle S \rangle = \langle a, b \rangle$ is dihedral and $S$ is locally maximal product-free in $H$. Let $o(ab) = m$, so $H \cong D_{2m}$.  The non-trivial coset of the subgroup $\langle ab \rangle$ is product-free of size $m$. So if $c$ lies in this coset, then we have $m = 3$ and $H \cong D_6$. If $c$ does not lie in this coset then $c = (ab)^i$ for some $i$, and from the relations in a dihedral group $ac^{-1} = ca$, $c^{-1}a = ac$, $bc^{-1} = cb$ and $c^{-1}b = bc$. The coset $\langle ab \rangle a$ consists of $m$ involutions, which cannot lie in $\sqrt S$. Thus $\langle ab \rangle a \subseteq T(S)$ by Lemma \ref{3.1}. 
A straightforward calculation shows that  
\begin{align*}
\langle ab \rangle a = T(S) \cap \langle ab \rangle a &= \{a,b,ac,ca,bc,cb,ac^{-1}, c^{-1}a, bc^{-1}, c^{-1}b\}\\
&= \{a,b,ac,ca,bc,cb\}
\end{align*}
This means $m \leq 6$, and $S$ consists of two generating involutions $a, b$ plus a power of their product $ab$, with the property that any two-element subset of $S$ generates $\langle a, b\rangle$. A glance at Table \ref{tab1} shows there are no locally maximal product-free sets of this form in $D_{2m}$ for $m \leq 6$. Therefore the only possibility is that $\langle S \rangle \cong D_6$, with $S$ consisting of the three reflections in $\langle S \rangle$.  By Theorem \ref{key}(i), the index of $\langle S \rangle$ in $G$ is a power of $2$. By Theorem \ref{key}(ii), $|G| \leq 2|T(S)|\cdot |\langle S\rangle|$. Thus $|G| \in \{6,12,24,48\}$. Suppose for contradiction that $|G| = 48$. Now $G = T(S) \cup \sqrt S$, and since $S$ consists of involutions, the elements of $\sqrt S$ have order 4. So $G$ contains two elements of order $3$, three elements of order 2 and the remaining non-identity elements have order $4$. Then the $46$ elements of $G$ whose order is a power of 2 must lie in three Sylow $2$-subgroups of order $16$, with trivial pairwise intersection. Each of these groups therefore has a unique involution and $14$ elements of order $4$, all of which square to the given involution.  But no group of order $16$ has fourteen elements of order $4$. Hence $|G| \neq 48$, and so $|G| \leq 24$. Therefore either $|G| \leq 24$ or $G$ contains exactly one involution.
\end{proof}

Before we establish the next result, we first make a useful observation. Suppose $S =\{a, b, c\}$ where $a, b, c \in G$ and $c$ is an involution. Then a straightforward calculation shows that 

\begin{equation}
\label{ts}
T(S) \subseteq \left\{\begin{array}{cc}1,a,b,c,a^2,b^2,ab,ba,ac,ca,bc,cb,\\
ab^{-1}, ba^{-1}, ca^{-1}, cb^{-1}, a^{-1}b, a^{-1}c, b^{-1}a, b^{-1}c
\end{array}
\right\}.
\end{equation}

\begin{lemma}\label{order3} Suppose $S$ is a locally maximal product-free set of size $3$ in $G$, every $2$-element subset of $S$ generates $\langle S\rangle$, and $S$ contains exactly one involution. Then either $|G| \leq 24$ or $S = \{a,b,c\}$, where $a$ and $b$ have order $3$ and $c$ is an involution. 
\end{lemma}

\begin{proof} Suppose $S = \{a,b,c\}$ where $c$ is an involution and $a, b$ are not. Consider $a^{-1}$. Recall that $G = T(S) \cup \sqrt S$. If $a^{-1} \in \sqrt S$ then $a^{-2} \in \{a,b,c\}$ which implies that either $a$ has order $3$ or $\langle S\rangle$ is cyclic (because for example if $a^{-2} = b$ then $\langle S\rangle = \langle a, b \rangle = \langle a\rangle$). Thus if $a^{-1} \in \sqrt S$ implies that either $a$ has order 3 or (by Lemma \ref{cyclic}) $|G| \leq 24$. Suppose then that $a^{-1} \in T(S)$. The elements of $T(S)$ are given in Equation \ref{ts}. 
If $a^{-1} \in \{b,b^2,ab,ba,ab^{-1}, ba^{-1}, a^{-1}b, b^{-1}a\}$ then by remembering that $\langle S \rangle = \langle a, b\rangle$, we deduce that $\langle S \rangle$ is cyclic, generated  by either $a$ or $b$. For example, $a^{-1} = ba$ implies $b \in \langle a\rangle$. Similarly, if $a^{-1} \in \{c,ac, ca, a^{-1}c, c^{-1}a\}$, then $\langle S \rangle$ is cyclic. Since $a$ has order at least 3, we cannot have $a^{-1} \in \{1,a\}$. If $a^{-1} \in \{bc,cb,b^{-1}c, c^{-1}b\}$, then $S$ would not be product-free. For instance $a^{-1} = b^{-1}c$ implies that $b^{-1}ca = 1$, and hence $ac = b$. The only remaining possibility is $a^{-1} = a^2$, meaning that $a$ has order 3. The same argument with $b^{-1}$ shows that $b$ also has order $3$.
\end{proof}

We can now prove Theorem \ref{msf3}, which states that if $S$ is a locally maximal product-free set of size 3 in a group $G$, such that every two element subset of $S$ generates $\langle S\rangle$, then $|G| \leq 24$.

\paragraph{Proof of Theorem \ref{msf3}}
Suppose $S$ is a locally maximal product-free set of size 3 in $G$ such that every two element subset of $S$ generates $\langle S\rangle$. Then by Lemma \ref{order3}, either $|G| \leq 24$ or $S = \{a,b,c\}$ where $a$ and $b$ have order $3$ and $c$ is an involution. In the latter case, we observe that $aca^{-1}$ is an involution, so must be contained in $T(S)$. Using Equation \ref{ts} we work through the possibilities. Obviously it is impossible for $aca^{-1}$ to be equal to any of $1, a, b, a^2$ or $b^2$ because these elements are not of order $2$. If any of $ac, ca, a^{-1}c, c^{-1}a, bc, cb, b^{-1}c$ or $cb^{-1}$ were involutions, then it would imply that $\langle S\rangle$ was generated by two involutions whose product has order 3. For example if $ac$ were an involution then $\langle c, ac\rangle = \langle a,c\rangle = \langle S \rangle$. That is, $\langle S\rangle$ would be dihedral of order $6$. But there is no product-free set in $D_6$ containing two elements of order 3, because if $x, y$ are the elements of order 3 in $D_6$ then $x^2 = y$ and $y^2 = x$. So the remaining possibilities for $aca^{-1}$ are $c, ab, ba, ab^{-1}, ba^{-1}, a^{-1}b$ and $b^{-1}a$. Now $aca^{-1} = ab$ implies $c = ba$, whereas $aca^{-1} = ab^{-1}$ implies $bc = a$ and $aca^{-1} = ba^{-1}$ implies $b = ac$, each of which contradicts the fact that $S$ is product-free. We are now left with the cases $aca^{-1} = c$, $aca^{-1} = ba$ and $aca^{-1} = a^{-1}b$ (which, if it is an involution, equals $b^{-1}a$). If $aca^{-1} = c$, then $\langle S \rangle = \langle a, c\rangle = C_6$, but the only product-free set of size 3 in $C_6$ contains no elements of order 3, so this is impossible. Therefore $aca^{-1} \in \{ba, a^{-1}b\}$. If $aca^{-1} = ba$, then $a^{-1}ba = ca^{-1}$, so $ac = a^{-1}b^{-1}a$, which has order 3. If $aca^{-1} = a^{-1}b$, then $ac = a^{-1}ba$, again of order 3. So we see that 
$$\langle S \rangle = \langle a, c : a^3 = 1, c^2 = 1, (ac)^3 = 1\rangle.$$
This is a well known presentation of the alternating group $\mathrm{Alt}(4)$. As $c$ is the only element of $S$ whose order is even, we see that $|\hat S| \leq 1$, and hence $|G| \leq 2|\mathrm{Alt}(4)| = 24$. Therefore in all cases $|G| \leq 24$.\qed

\section{Data and Programs}
Though Table \ref{tab1} is essentially just Table 5 from \cite{gh}, we have taken the opportunity here to correct a typographical error in the entry for the (un-named) group of order $16$. We  provide below the GAP programs used to obtain the table.

\begin{prog}
A program that tests if a set T is product-free.
\begin{verbatim}
## It returns "0" if T is product-free, and "1" if otherwise.
prodtest:= function(T)
	local  x, y, prod;
	prod:=0;
	for x in T do
	    for y in T do
	      if x*y in T then
	         prod:=1;
	      fi;
	    od;
	od; 
	return prod;
end;
\end{verbatim}
\end{prog}

\begin{prog}
A program for finding all locally maximal product-free sets of size $3$ in $G$.
\begin{verbatim}
##It prints the list of all locally maximal product-free sets of size 3 in G.
LMPFS3:=function(G)
local L, lmpf, combs, x, pf, H, y, z, s, i, q;
L:=AsSortedList(G); lmpf:=[]; combs:=Combinations(L,3);
for i in [1..Binomial(Size(L),3)] do
  pf:=combs[i]; 
  if prodtest(pf)=0 then
   s:=Size(lmpf); H:=Difference(L,pf);
   for y in [1..3] do
     for z in [1..3] do
       H:=Difference(H, [pf[y]*pf[z], pf[y]*(pf[z])^-1, ((pf[y])^-1)*pf[z]]);
     od; 
   od;
   for q in L do
       if q^2 in pf then
          H:=Difference(H, [q]);
       fi;
   od;	
   if Size(H) = 0 then
      lmpf:=Union(lmpf, [pf]);
   fi;
  fi; 
od;
if Size(lmpf) > 0 then 
  Print(G,"\n",L,"\n","Structure Description of G is ",StructureDescription(G),
  "\n", "Gap Id of G is ", IdGroup(G), "\n", "\n", lmpf, "\n", "\n");
fi;
end;
\end{verbatim}
\end{prog}

\begin{landscape}
\begin{table}
\begin{tabular}{ll|c|l|c}
$G$ && $S$ & $\la S \ra$ & \begin{tabular}{c}\# Locally maximal \\product-free sets \\of size $3$ in $G$\end{tabular}\\
 \hline
$\langle g: g^6 = 1\rangle$ & $\cong C_6$ & %
$\{g, g^3, g^5  \}$ & $\cong C_6$ &1 \\
$\langle g, h: g^3 = h^2 = 1, hgh = g^{-1} \rangle$ & $\cong D_6$
& $\{h, gh, g^2h\}$ & $\cong D_6$ & 1\\

$\langle g: g^8 = 1\rangle$ & $\cong C_8$ & %
$\{g, g^{-1}, g^4\}$ & $\cong C_8$ & 2\\
$\langle g, h: g^4 = h^2 = 1, hgh^{-1} = g^{-1}\rangle$ & $\cong D_8$ & %
$\{h, gh, g^2  \}$ & $\cong D_8$ & 4\\
$\langle g: g^9 = 1\rangle$ & $\cong C_9$ & %
$\{g, g^3, g^8  \}, \{g, g^4, g^7\}$ & $\cong C_9$ & 8\\
$\langle g, h: g^3=h^3=1, gh=hg\rangle$ & $\cong C_3 \times C_3$ & %
$\{g, h, g^2h^2  \}$ & $\cong C_3 \times C_3$ & 8\\
$\langle g: g^{10} = 1\rangle$ & $\cong C_{10}$ & %
$\{g^2, g^5, g^8\}, \{g, g^5, g^8\}$ & $\cong C_{10}$ & 6\\
$\langle g: g^{11} = 1\rangle$ & $\cong C_{11}$ & %
$\{ g, g^3, g^5 \}$ & $\cong C_{11}$ & 10\\
$\langle g: g^{12} = 1\rangle$ & $\cong C_{12}$ & %
$\{g^2, g^6, g^{10}  \}$ & $\cong C_6$ & 1\\
 &  & %
$\{g, g^6, g^{10}\}, \{g, g^3, g^8\}$ & $\cong C_{12}$ & 8\\
$\langle g, h: g^6 = 1, g^3 = h^2, hgh^{-1} = g^{-1}\rangle$ & $\cong Q_{12}$ & %
$\{g, g^3, g^5  \}$ & $\cong C_6$ & 1\\
Alternating group of degree 4 & $=$ Alt(4)& %
$\{x, y, z: x^2 = y^2 = z^3 = 1\}$ & $\cong$ Alt(4) & 48\\
&& $\{x, z, xzx: x^2 = z^3=1\}$ &&\\
&&$\{x, z, zxz: x^2 = z^3 = 1\}$&&\\
$\langle g: g^{13} = 1\rangle$ & $\cong C_{13}$ & %
$\{g, g^3, g^9\}, \{g, g^6, g^{10}\}$ & $\cong C_{13}$ & 16\\
$\langle g: g^{15} = 1\rangle$ & $\cong C_{15}$ & %
$\{g, g^3, g^{11}  \}$ & $\cong C_{15}$ & 4\\
$\langle g, h: g^4 = h^4 = 1, gh = hg\rangle$ & $\cong C_4\times C_4$ & %
$\{g, h, g^{-1}h^{-1}  \}$ & $\cong C_4\times C_4$ & 16\\
$\langle g, h: g^8=1, g^4=h^2, hgh^{-1} = g^{-1}\rangle$ & $\cong Q_{16}$ & %
$\{g, g^4, g^{-1}  \}$ & $\cong C_8$ & 2\\
$\langle g, h: g^8 = h^2 = 1, hgh^{-1} = g^5\rangle$ & (order 16) & %
$\{g, g^6, g^3h\}$ & $\cong G$ & 8\\
$\langle g, h: g^{10} = 1, g^5=h^2, hgh^{-1} = g^{-1}\rangle$ & $\cong Q_{20}$ & %
$\{g, g^5, g^8\}, \{g^2, g^5, g^8\}$ & $\cong C_{10}$ & 6\\
$\langle g, h: g^3 = h^7 = 1, ghg^{-1} = h^2\rangle$ & $\cong C_7\rtimes C_3$ & %
$\{gh, gh^{-1}, g^{-1}\}$ & $\cong C_7\rtimes C_3$ & 42\\
$\langle x : x^3 = 1\rangle \times \la g, h:  g^4 = 1, g^2 = h^2, hgh^{-1} = g^{-1}\rangle$ & $\cong C_3\times Q_8$ & %
$\{g^2, xg^2, x^2g^2\}$ & $\cong C_6$ & 1\\
$\langle g, h: g^{12} = 1, g^6 = h^2, hgh^{-1} = g^{-1}\rangle$ & $\cong Q_{24}$ & %
$\{g^2, g^6, g^{10}\}$ & $\cong C_6$ & 1\\
 &  & %
$\{g, g^6, g^{10}  \}$ & $\cong C_{12}$ & 4\\
\end{tabular}
\caption{Locally maximal product-free sets of size $3$ in groups of order up to $24$}
\label{tab1}
\end{table}
\end{landscape}


\begin{thebibliography}{99}
\bibitem{babaisos}
L\'aszl\'o Babai and Vera T.~S\'os, {\em Sidon sets in groups and
induced subgraphs of Cayley graphs}, European J. Combin. 6 (1985),
101--114.


  \bibitem{GAP4}
  The GAP~Group, \emph{GAP -- Groups, Algorithms, and Programming, 
  Version 4.7.7}; 
  2015,
  ({\tt{http://www.gap-system.org}}).

\bibitem{gh} Michael Giudici and Sarah Hart, {\em Small maximal sum-free sets},  Elect. J. Comb. 16 (2009), 1--17.
\bibitem{gowers}
W.~T.~Gowers, {\em Quasirandom groups}, Combin. Probab. Comput. 17
(2008), no. 3, 363--387.
%
\bibitem{greenruzsa}
Ben Green and Imre Z. Ruzsa, {\em Sum-free sets in abelian
groups}, Israel J. Math. 147 (2005), 157--188.
\bibitem{kedlaya97}
Kiran S.~Kedlaya, {\em Large product-free subsets of finite
groups}, J. Combin. Theory Ser. A  77  (1997),  339--343.

\bibitem{streetwhite}
A.~P.~Street and E.~G.~Whitehead Jr., {\em Group Ramsey theory},
J. Combinatorial Theory Ser. A  17  (1974), 219--226.
%

\bibitem{SW1974A}
A. P. Street and E. G. Whitehead, Jr., Sum-free sets, difference sets and cyclotomy.
Combinatorial Math., Lecture notes in Mathematics 403 (1974), 109--124.
\end{thebibliography}
\end{document}